\documentclass[10pt,twoside]{amsart}
\usepackage{amsfonts}
\usepackage{amssymb}
\usepackage{graphicx,color}
\usepackage{a4wide}
\usepackage{amsmath}
\usepackage{amsthm}

\date{\today}

\newtheorem{theorem}{Theorem}[section]
\newtheorem{lemma}[theorem]{Lemma}

\theoremstyle{definition}
\newtheorem{definition}[theorem]{Definition}
\newtheorem{example}[theorem]{Example}
\newtheorem{remark}[theorem]{Remark}

\title{{\bf Linear programming bounds for covering radius of spherical designs }}

\author{}
\begin{document}
\maketitle

\vspace*{-8mm}

\noindent
Peter Boyvalenkov,
Institute of Mathematics and Informatics,
Bulgarian Academy of Sciences,
8 G.Bonchev str., 1113 Sofia, BULGARIA,
{\tt peter@math.bas.bg} \\[2pt]
Maya Stoyanova,
Faculty of Mathematics and Informatics,
Sofia University ''St. Kliment Ohridski''
5 James Baucher blvd, Sofia, BULGARIA, {\tt
stoyanova@fmi.uni-sofia.bg}  \\[2pt]

\begin{abstract} \noindent
We apply polynomial techniques (linear programming) to obtain lower and upper bounds on the covering
radius of spherical designs as function of their dimension, strength, and cardinality.
In terms of inner products we improve the lower bounds due to Fazekas and Levenshtein and propose new upper bounds.
Our approach to the lower bounds involves certain signed measures whose corresponding series of orthogonal polynomials
are positive definite up to a certain (appropriate) degree. Upper bounds are based on a geometric observation
and more or less standard linear programming techniques.
%\subclass{05B30}
\end{abstract}
\setcounter{equation}{0}

\section{Introduction}
\label{intro}
Spherical designs were introduced in 1977 by Delsarte-Goethals-Seidel \cite{DGS}.

\begin{definition} \label{def-designs}
A spherical $\tau$-design $C \subset \mathbb{S}^{n-1}$ is a finite subset of $\mathbb{S}^{n-1}$ such that
\[ \frac{1}{\mu(\mathbb{S}^{n-1})} \int_{\mathbb{S}^{n-1}} f(x) d\mu(x) =
                  \frac{1}{|C|} \sum_{x \in C} f(x) \]
($\mu(x)$ is the Lebesgue measure) holds for all polynomials $f(x) = f(x_1,x_2,\ldots,x_n)$ of degree at most $\tau$
(i.e. the average of $f$ over the set is equal to the average of $f$ over $\mathbb{S}^{n-1}$).
\end{definition}

The maximal possible $\tau=\tau(C)$ is called strength of $C$. It is convenient to assume that the vectors of
$C$ span $\mathbb{R}^n$ (otherwise we can reduce the dimension as necessary).

\begin{definition} \label{def-covrad}
Let $C \subset \mathbb{S}^{n-1}$ be a finite set (spherical design in our applications).
For a fixed point $y \in \mathbb{S}^{n-1}$ the distance between $y$ and $C$ is defined in the usual way by
\[ d(y,C) := \min \{d(y,x): x\in C\}. \]
Then the covering radius of $C$ is
\[ r(C) := \max \{d(y,C): y \in \mathbb{S}^{n-1}\}. \]
\end{definition}

We consider the equivalent quantity
\[ \rho(C) := 1 - \frac{r^2(C)}{2} = \min_{y \in \mathbb{S}^{n-1}} \max_{x \in C} \{ \langle x,y \rangle \}. \]
For fixed dimension $n$, strength $\tau$ and cardinality $|C|$ we obtain estimations for $\rho(C)$ by polynomial techniques.
The following equivalent definition is an important source for investigations of the structure of spherical designs (see \cite{FL,BBD,BBKS}).

\begin{definition} \label{def-des}
A code $C \subset \mathbb{S}^{n-1}$ is a spherical $\tau$-design if and only if for any point $y \in \mathbb{S}^{n-1}$ and any real
polynomial $f(t)$ of degree at most $\tau$, the equality
\begin{equation}
\label{defin_f}
\sum_{x \in C}f(\langle x,y \rangle ) = f_0|C|
\end{equation}
holds, where $f_0$ is the first coefficient in the Gegenbauer expansion $f(t)=\sum_{i=0}^k f_i P_i^{(n)}(t)$.
Here $P_i^{(n)}(t)$, $i \geq 0$, are Gegenbauer polynomials \cite{Sze} normalized by $P_i^{(n)}(1)=1$.
\end{definition}

Linear programming bounds for covering radius of spherical designs were obtained by Fazekas and Levenshtein \cite[Theorem 2]{FL}
(see also \cite{Yud95}). They prove that if $C$ is a $(2k-1+e)$-design, $e \in \{0,1\}$, then
\begin{equation}
\label{FL_bound}
\rho(C) \geq t_{FL} = t_k^{0,e},
\end{equation}
where $t_k^{0,e}$ is the largest zero of the Jacobi polynomial $P^{(\alpha,\beta)}(t)$, $\alpha = \frac{n-3}{2}$,
$\beta = \frac{n-3}{2}+e$ \cite{Sze}. For example, \eqref{FL_bound} gives $R(n,4,M) \geq \frac{1+\sqrt{n+3}}{n+2}$
for every $M \geq D(n,4)=n(n+3)/2$. Note that the Fazekas-Levenshtein bound does not depend on the cardinality $M$.
The bound (\ref{FL_bound}) for even $\tau=2k$ is attained if $C$ is a tight spherical $(2k)$-design, i.e. a spherical $(2k)$-design on $\mathbb{S}^{n-1}$ of cardinality $D(n,2k)$ (see \eqref{DGS-bound} below). In this case the covering radius is realized by each point of $\mathbb{S}^{n-1}$ whis is antipodal to a point of $C$.

We remark that tight spherical $(2k)$-designs, $\tau \geq 4$, could possibly exist for $\tau=2$ and 4 only (see \cite{BD1,BD2,BMV04,NV12}). If $C \subset \mathbb{S}^{n-1}$ is a tight spherical $4$-design, then $n=(2m+1)^2-3$,
where $m$ is a positive integer. Examples are known for $m=1$ and $m=2$ only.

In this paper we obtain lower and upper bounds on the covering radius $\rho(C)$ of designs with fixed dimension $n$, even strength $2k$,
and cardinality. While the Fazekas-Levenshtein bound \eqref{FL_bound} depends on the strength only, our bounds are also function of the cardinality of the designs of considered dimension and strength. To this end, we apply in appropriate combinations polynomials
in Definition \ref{def-des}. Our approach for lower bounds requires certain signed measure which is positive definite up to
certain degree and therefore defines a sequence of orthogonal polynomials which appears good for our purposes.
Finally in this description, we underline that we investigate the structure of $C$ with respect to point(s) where the
covering radius $\rho(C)$ is attained. Our upper bounds techniques is applied for odd strengths as well.

In Section \ref{sec:1} we introduce the Delsarte-Goethals-Seidel bound, some notation and the required signed measure.
The properties of the generated series of orthogonal polynomials are studied in Section \ref{sec:2}. Section \ref{sec:3} is devoted to
the lower bounds where different techniques are applied depending how close to $-1$ is the smallest inner product
of a point where the covering radius is attained. Upper bounds are obtained in Section \ref{sec:4}. Examples of new bounds are
shown in tables in Sections \ref{sec:3} and \ref{sec:4}.

\section{Preliminaries}
\label{sec:1}

\subsection{Delsarte-Goethals-Seidel bound}
\label{sec:1.1}

For fixed dimension $n \geq 2$ and strength $\tau \geq 1$ the minimum cardinality of
a spherical $\tau$-design $C \subset \mathbb{S}^{n-1}$ is bounded from below by
Delsarte-Goethals-Seidel \cite{DGS} as follows:
\begin{equation}
\label{DGS-bound}
|C| \geq D(n,\tau) := {n+k-2+e \choose n-1}+{n+k-2 \choose n-1},
\end{equation}
where $\tau=2k-1+e$, $e \in \{0,1\}$. This bound \eqref{DGS-bound} is rarely attained.
Improvements were obtained \cite{Boy1,Yud97,NN}
but it is still unknown, for example, if there exist spherical 4-designs of 10 points on $\mathbb{S}^2$.
On the other hand, it was shown by Bondarenko, Radchenko and Viazovska \cite{BRV13,BRV15} that for fixed
dimension $n$ and strength $\tau$ there exist spherical $\tau$-designs on $\mathbb{S}^{n-1}$ for any
cardinality $N \geq C_n \tau^{n-1}$, where the constant $C_n$ depends on the dimension $n$ only.

\subsection{Notations for the structure of spherical designs}
\label{sec:1.2}

We are interested in the structure of designs whose cardinality is close to the bound \eqref{DGS-bound}.
Apart from the interest here, we believe that our approach could be useful for proving nonexistence
results (see \cite{BBD}).

Let $C \in \mathbb{S}^{n-1}$ be a spherical design. For arbitrary point $y \in \mathbb{S}^{n-1}$, we consider the (multi)set
\[ I(y) = \{ \langle x,y \rangle : x \in C\}
     = \{ t_1(y),t_2(y),\ldots,t_{|C|}(y) \}, \]
where we order $I(y)$ by $-1 \leq t_1(y) \leq t_2(y) \leq \cdots \leq t_{|C|}(y) \leq
1$. Note that $t_{|C|}(y)=1-d^2(y,C)/2$ and, in particular, $t_{|C|}(y)=1 \iff y \in C$.
In what follows we always assume that $y$ is a point on $\mathbb{S}^{n-1}$
where the covering radius is realized, in particular $t_{|C|}(y)=\rho(C)$ (see Lemma \ref{last-n} below for a stronger result).

The tight spherical $2k$-designs have $t_1(y)=-1$ for every $y$ realizing the covering radius \cite[Theorem 3]{FL}. Therefore, it is natural and important to
investigate how close is $t_1(y)$ to $-1$ at least for cardinalities which are close to the Delsarte-Goethals-Seidel bound $D(n,2k)$.
We find convenient to consider two cases: $t_1(y) \in [-1,\ell]$ and $t_1(y) \geq \ell$ for some $\ell > -1$.

We will utilize the following geometric observation on the structure of $I(y)$.

\begin{lemma} \label{last-n}
If $y$ is a point on $\mathbb{S}^{n-1}$ where
the covering radius is realized, then
\[ t_{|C|}(y)=t_{|C|-1}(y)=\cdots=t_{|C|-n+1}(y)=\rho(C). \]
\end{lemma}

\begin{proof} The point $y$ is the center of a spherical cap
which corresponds to one of the facets ($(n-1)$-dimensional faces) of the
convex hull of $C$. Since every facet has at least $n$ vertices,
the statement follows.
\end{proof}

\subsection{A signed measure and corresponding orthogonal polynomials}
\label{sec:1.3}

As mentioned above, by $P_i^{(n)}(t)$ we denote the Gegenbauer polynomials \cite{Sze} of degrees $i=0,1,\ldots$, normalized by $P_i^{(n)}(1)=1$.
The measure of orthogonality of Gegenbauer polynomials is
\[ d\mu(t) := c_n (1-t^2)^{\frac{n-3}{2}}\, dt, \ \ \ t \in [-1,1], \ \ \
 c_n := \Gamma(\frac{n}{2})/\sqrt{\pi}\Gamma(\frac{n-1}{2}). \]

Recall that a signed Borel measure $\nu$ on $\mathbb{R}$ for which all polynomials are integrable is called
{\it positive definite up to degree $m$} if $\int p^2(t) d \nu(t) > 0$
for all real nonzero polynomials $p(t)$ of degree at most $m$.

It was proved in \cite[Lemma 2.2]{BDHSS-DCC2019} that the signed measure
\[ d \mu_\ell(t) := c_{n,\ell} (t-\ell) d \mu(t), \ \ \ t \in [-1,1], \ \ \
 c_{n,\ell} := - 1/\ell. \]
is positive definite up to degree $k-1$ provided that $\ell < t_{k,1}$, where $t_{k,1}$ is the smallest zero of the Gegenbauer
polynomial $P_k^{(n)}(t)$.
This implies (see Corollary 2.3 in \cite{BDHSS-DCC2019}) the existence of a finite sequence of polynomials
$(P_i^{0,\ell}(t))_{i=0}^k$ which are orthogonal with respect to $d \mu_\ell(t)$.
Moreover, with the normalization $P_i^{0,\ell}(1)=1$ these polynomials are uniquely determined by the
Gram-Schmidt orthogonalization. This allows us to write (see \eqref{Ch-D} below) explicitly
$P_i^{0,\ell}(t)$ via the Christoffel-Darboux kernel
\[ T_i(u,v) := \sum_{j=0}^i r_j P_j^{(n)}(u) P_j^{(n)}(v). \]

The boundary case $\ell=-1$ leads to polynomials which Levenshtein \cite{Lev2} denoted by $P_i^{0,1}(t)$ and
which are in fact (normalized) Jacobi polynomials with parameters $(\alpha,\beta) = ((n-3)/2,(n-1)/2)$.

The parameter $k$ henceforth comes from $\tau=2k$, the strength of designs under consideration.

\section{Properties of the polynomials $P_i^{0,\ell} (t)$}
\label{sec:2}

\subsection{Interlacing of roots}
\label{sec:2.1}

The representation (verified in the next theorem)
\begin{equation} \label{Ch-D}
P_i^{0,\ell} (t) = \frac{T_i (t,\ell)}{T_i(1,\ell)} = \frac{(1-\ell)\left( P_{i+1}^{(n)} (t) - P_i^{(n)} (t)P_{i+1}^{(n)} (\ell)/
P_{i}^{(n)} (\ell) \right)}{(t-\ell)\left(1-P_{i+1}^{(n)} (\ell)/P_{i}^{(n)} (\ell)\right)}
\end{equation}
via the  Christoffel-Darboux formula allows us to derive interlacing properties of the zeros of $P_i^{0,\ell}(t)$ with respect to the
zeros of $P_j^{(n)}(t)$, $j=i,i+1$.

We denote by
\[ t_{i,1} < t_{i,2} < \cdots < t_{i,i} \]
the zeros of $P_i^{(n)}(t)$. For the zeros of $P_i^{0,\ell}(t)$ we use the same idea of notation with just adding upper indices $0,\ell$.

\begin{theorem}
\label{roots-of-p0-ell}
Let $\ell$ and $k$ be such that $t_{k+1,1} < \ell < t_{k,1}$ and
$P_{k+1}^{(n)}(\ell)/P_{k}^{(n)}(\ell)<1$. Then
\begin{equation} \label{Pi0-ell}
P_i^{0,\ell}(t) = \frac{T_i (t,\ell)}{T_i (1,\ell)} = m_i^{0,\ell} t^i + \cdots,\quad i=0,1,\dots, k,
\end{equation}
with $m_i^{0,\ell}>0$ and all zeros $\{t_{i,j}^{0,\ell}\}_{j=1}^i$ of $P_i^{0,\ell}(t)$
are in the interval $[\ell,1]$. Moreover, the interlacing rules
\begin{equation}\label{Interlacing}\begin{split}
t_{i,j}^{0,\ell} &\in (t_{i,j}, t_{i+1,j+1}), \ i=1,\dots, k-1, j=1,\dots, i ;\\
t_{k,j}^{0,\ell} &\in (t_{k+1,j+1}, t_{k,j+1}), \ j=1,\dots,k-1, \ t_{k,k}^{0,\ell}\in (t_{k+1,k+1},1),
\end{split}
\end{equation}
hold.
\end{theorem}

\begin{proof}
It follows from the Christoffel-Darboux formula in \eqref{Ch-D} that $(t-\ell)T_i(t,\ell)$ is
a linear combination of the polynomials $P_{i+1}^{(n)}(t)$ and  $P_i^{(n)}(t)$. This immediately implies that
the polynomial
$T_i(t,\ell)$ is orthogonal to any polynomial of degree at most $i-1$ with respect to the measure $d\mu_{\ell}(t)$. Now
\eqref{Pi0-ell} follows from the positive definiteness of $d\mu_{\ell}(t)$ up to degree $k-1$,
the uniqueness of the Gram-Schmidt orthogonalization process and the normalization. The comparison of coefficients
in \eqref{Pi0-ell} shows that $m_i^{0,\ell}>0$, $i=0,1,\ldots,k$.

We proceed with the proof of \eqref{Interlacing}, the location of the zeros of $P_i^{0,\ell}(t)$. It follows from \eqref{Pi0-ell} and \eqref{Ch-D} that they are solutions of the equation
\begin{equation} \label{FracEq}
\frac{P_{i+1}^{(n)}(t)}{P_i^{(n)}(t)} = \frac{P_{i+1}^{(n)}(\ell)}{P_i^{(n)}(\ell)}.
\end{equation}

Let $i<k$. Then the zeros of $P_{i+1}^{(n)}(t)$ and $P_i^{(n)}(t)$ are interlaced and contained in $[t_{k,1},t_{k,k}]$. Since
${\rm sign}\,P_i^{(n)}(\ell)=(-1)^i$, the right hand side of \eqref{FracEq} is a negative constant. The left hand side $P_{i+1}^{(n)}(t)/P_i^{(n)}(t)$
has simple poles at $t_{i, j}$, $j=1,\dots, i$, and simple zeros at $t_{i+1,j}$, $j=1,\dots, i+1$. Therefore, there is at least one solution $t_{i,j}^{0,\ell}$ of \eqref{FracEq} on every subinterval $(t_{i,j}, t_{i+1,j+1})$, $j=1,\dots,i$, which accounts for all zeros of $P_i^{0,\ell}(t)$.

If $i=k$, then it follows from $t_{k+1,1}<\ell<t_{k,1}$ that $P_{k+1}^{(n)}(\ell)/P_k^{(n)}(\ell)>0$. Moreover, we can account as above
for the first $k-1$ solutions of \eqref{FracEq}, namely \[t_{k,j}^{0,\ell} \in (t_{k+1,j+1}, t_{k,j+1}), \ j=1,\dots,k-1.\]
For the last zero of $P_{k}^{0,\ell}(t)$ we utilize the fact that $P_{k+1}^{(n)}(t)/P_k^{(n)}(t)>0$ for $t\in (t_{k+1,k+1},\infty)$. As $\lim_{t\to \infty}P_{k+1}^{(n)}(t)/P_k^{(n)}(t) = \infty$, we have one more solution $t_{k,k}^{0,\ell}>t_{k+1,k+1}$ of \eqref{FracEq}.
Finally, since $P_{k+1}^{(n)}(\ell)/P_{k}^{(n)}(\ell)<1=P_{k+1}^{(n)}(1)/P_k^{(n)}(1)$ by assumption, we conclude that $t_{k,k}^{0,\ell}<1$.
\end{proof}

\begin{remark}
In general, the polynomial $P_k^{0,\ell} (t)$ is well defined for all $t_{k+1,1}<\ell<t_{k,1}$, but its largest
root leaves the interval $[-1,1]$ and the leading coefficient becomes negative when (as the proof above shows) the condition
$P_{k+1}^{(n)}(\ell)/P_{k}^{(n)}(\ell)<1$ is not satisfied.
\end{remark}

We next describe the three-term recurrence relation for the polynomials $P_i^{0,\ell}(t)$. The positive definiteness of the measure $d\mu_{\ell}(t)$ implies that
\[r_i^{0,\ell} := \left( \int_{-1}^1 \left( P_i^{0,\ell} (t)\right)^2 \,d\mu_{\ell} (t) \right)^{-1} >0, \  \ i=0,1,\dots,k-1. \]
Then the three-term recurrence relation for $P_i^{0,\ell}(t)$ can be written as
\[ (t-a_i^{0,\ell}) P_i^{0,\ell} (t) = b_i^{0,\ell} P_{i+1}^{0,\ell} (t) + c_i^{0,\ell} P_{i-1}^{0,\ell} (t), \]
$i=1,2, \dots ,k-1$, where the coefficients are given by
\[ b_i^{0,\ell} = \frac{m_{i+1}^{0,\ell}}{m_i^{0,\ell}}>0, \ c_i^{0,\ell} = \frac{r_{i-1}^{0,\ell} b_{i-1}^{0,\ell}}{r_i^{0,\ell}}>0, \
a_i^{0,\ell} = 1 - b_i^{0,\ell} - c_i^{0,\ell}. \]
The initial conditions are $P_0^{0,\ell}(t)=1$ and $P_1^{0,\ell}(t) = (n\ell t + 1)/(n\ell + 1)$.
We also note that the orthogonality implies that the zeros of the polynomials $P_j^{0,\ell} (t)$ interlace; i.e.,
\[ t_{j,i}^{0,\ell} < t_{j-1,i}^{0,\ell} < t_{j,i+1}^{0,\ell}, \ i=1,2,\dots, j-1. \]

\subsection{A quadrature formula}
\label{sec:2.2}

The next theorem involves the zeros of $P_{k}^{0,\ell}(t)$ in a positive quadrature formula of Gauss-Jacobi type.
We denote by $L_i (t)$, $i=0,1,\ldots,k$, the Lagrange basic polynomials generated by the
nodes $\ell<t_{k,1}^{0,\ell}<t_{k,2}^{0,\ell}<\cdots<t_{k,k}^{0,\ell}$ and set
\[ \theta_i:=\int_{-1}^1 L_i(t) d\mu(t), \ \ i=0,1,\dots,k. \]

\begin{theorem} \label{QFtheorem} Let $t_{k,1}^{0,\ell}<t_{k,2}^{0,\ell}<\cdots<t_{k,k}^{0,\ell}$
 be the zeros of the polynomial $P_{k}^{0,\ell}(t)$. Then the quadrature formula
\begin{equation}
\label{QF}
f_0 = \int_{-1}^1 f(t) d \mu(t)
    =\theta_0 f(\ell)+ \sum_{i=1}^{k} \theta_i f(t_{k,i}^{0,\ell})
\end{equation}
is exact for all polynomials of degree at most $2k$ and has positive weights $\theta_i>0$, $i=0,1,\dots, k$.
\end{theorem}

\begin{proof}
We first observe that \eqref{QF} is exact for the Lagrange basis at $k+1$ nodes
as defined above and hence for all polynomials of degree at most $k$. Indeed, any such polynomial can be written
as
\[ f(t)=f(\ell)L_0(t)+\sum_{i=1}^k f(t_{k,i}^{0,\ell}) L_i(t) \]
and integration over $[-1,1]$ with respect to $\mu(t)$ gives the result.

Given a polynomial $f(t)$ of degree at most $2k$, we divide it by $(t-\ell)P_k^{0,\ell}(t)$ and write
\begin{equation} \label{div}
f(t)=(t-\ell)P_{k}^{0,\ell} (t)q(t) + r(t),
\end{equation}
where $q(t)$ has degree at most $k-1$ and $r(t)$ has degree at most $k$.
We again integrate over $[-1,1]$ with respect to $\mu(t)$ and use the orthogonality of $P_{k}^{0,\ell}(t)$ to all polynomials of degree at
most $k-1$ with respect to $d\mu_{\ell}(t)$ to see that
\[ f_0=r_0=\theta_0 r(\ell)+ \sum_{i=1}^{k} \theta_i r(t_{k,i}^{0,\ell})=\theta_0 f(\ell)+ \sum_{i=1}^{k} \theta_i f(t_{k,i}^{0,\ell}) \]
by \eqref{div}. Therefore \eqref{QF} holds true for for all polynomials of degree at most $2k$.

We now show the positivity of the weights $\theta_i$, $i=0,\dots, k$.
First we fix $i \in \{1,2,\ldots,k\}$ and substitute $f(t)=(t-\ell)\left( u_{i} (t)\right)^2$ in \eqref{QF}, where
\[ u_{i}(t)=\frac{P_{k}^{0,\ell}(t)}{t-t_{k,i}^{0,\ell}}, \ \ \deg(u_i)=k-1, \]
(i.e., $f(t)$ is of degree $2k-1$ and is eligible for exactness in \eqref{QF}) to obtain
\[ f_0 = \theta_i (t_{k,i}-\ell) \left(u_{i} (t_{k,i}^{0,\ell})\right)^2. \]
On the other hand, we have
\[ f_0=\int_{-1}^1 (t-\ell)\left( u_{i} (t)\right)^2 d \mu(t) = \frac{1}{c_{n,\ell}} \int \left( u_{i} (t)\right)^2 d \mu_\ell(t) > 0 \]
by the positive definiteness up to degree $k-1$ of the measure $\mu_\ell(t)$. Finally, we
use in \eqref{QF} the polynomial $f(t)=\left( P_k^{0,\ell}(t)\right)^2$ of degree $2k$ to obtain
\begin{eqnarray*}
\theta_0 f(\ell) = f_0 > 0,
\end{eqnarray*}
whence $\theta_0 >0$. This completes the proof.
\end{proof}

\section{Improving the Fazekas-Levenshtein bound for even strengths}
\label{sec:3}

Let $C \subset \mathbb{S}^{n-1}$ be a spherical $2k$-design of cardinality $|C|>D(n,2k)$. We recall that $y \in \mathbb{S}^{n-1}$
is a point which realizes the covering radius of $C$ (see Lemma \ref{last-n}). It is convenient to
use an additional parameter $\ell$ which is close to $-1$ and to consider two cases for $t_1(y)$, namely
$\ell \leq t_1(y)$ and $t_1(y) \in [-1,\ell]$.

\subsection{Case $\ell \leq t_1(y)$}
\label{sec:3.1}

Lower bounds $-1 < \ell \leq t_1(y)$ imply improvements of the Fasekas-Levenshtein bound $t_{FL}$ by
using in \eqref{defin_f} suitable polynomials and applying the quadrature formula \eqref{QF}.

\begin{theorem} \label{impr_fl}
Let $C \subset \mathbb{S}^{n-1}$ be a spherical $2k$-design and $y \in \mathbb{S}^{n-1}$
is a point which realizes the covering radius of $C$.
If $\ell \leq t_1(y)$, then
\[ \rho(C) \geq t_{k,k}^{0,\ell}. \]
\end{theorem}

\begin{proof}
We apply \eqref{defin_f} for $y$, $C$, and
\[ f(t) = (t-\ell)(t-\rho(C))\left(\frac{P_k^{0,\ell}(t)}{t-t_{k,k}^{0,\ell}}\right)^2. \]
Then
\[ f_0|C| = \sum_{j=1}^{|C|} f(t_j(y)) \leq 0 \]
since $t_j(y) \in [\ell,\rho(C)]$ for every $j=1,2,\ldots,|C|$. On the other hand,
\[ f_0|C| = |C|\left(\rho_0 f(\ell) + \sum_{i=1}^k \theta_i f(t_{k,i}^{0,\ell})\right) = |C|\theta_k f(t_{k,k}^{0,\ell}) \]
follows from the quadrature formula \eqref{QF}. Therefore
$f(t_{k,k}^{0,\ell}) \leq 0$, whence $\rho(C) \geq t_{k,k}^{0,\ell}$.
\end{proof}

In the boundary case $\ell=-1$ Theorem \ref{impr_fl} gives the Fasekas-Levenshtein bound $\rho(C) \geq t_k^{0,1}$.
Therefore, we have improvement of the Fasekas-Levenshtein bound whenever it is known (or it is presumed)
that $\ell \leq t_1(y)$ for a point $y$ where the covering radius is realized.
We present in Table \ref{tab:1} values of $t_{k,k}^{0,\ell}$ for different $\ell$, compared to the Fazekas-Levenshtein bound
$t_{FL} = t_k^{0,1}$.

\begin{table}
\begin{center}
\caption{ Some lower bounds $t_{k,k}^{0,\ell}$. }
\label{tab:1}
\begin{tabular}{|c|c|c|c|c|c|}
  \hline\noalign{\smallskip}
  Dimension & Cardinality & Strength & $\ell$ & Fazekas- & New lower \\
  $n$ & $|C|$ & $\tau = 2k$ & & Levenshtein & bound, if \\
  &&&& lower bound & $\ell \leq t_1(y)$  \\
  &&&& $\rho(C) \geq t_k^{0,1}$ & $\rho(C) \geq t_{k,k}^{0,\ell}$  \\
  \noalign{\smallskip}\hline\noalign{\smallskip}
  3 & 10 & 4 & -0.97 & 0.689897 & 0.694892  \\
  \noalign{\smallskip}\hline\noalign{\smallskip}
  3 & 10 & 4 & -0.95 & 0.689897 & 0.698664  \\
  \noalign{\smallskip}\hline\noalign{\smallskip}
  3 & 10 & 4 & -0.9  & 0.689897 & 0.710257  \\
  \noalign{\smallskip}\hline\noalign{\smallskip}
  4 & 15 & 4 & -0.97 & 0.607625 & 0.611772  \\
  \noalign{\smallskip}\hline\noalign{\smallskip}
  4 & 15 & 4 & -0.95 & 0.607625 & 0.614815  \\
  \noalign{\smallskip}\hline\noalign{\smallskip}
  4 & 15 & 4 & -0.9  & 0.607625 & 0.623682  \\
  \noalign{\smallskip}\hline\noalign{\smallskip}
  3 & 17 & 6 & -0.97 & 0.822824 & 0.825859  \\
  \noalign{\smallskip}\hline\noalign{\smallskip}
  3 & 17 & 6 & -0.95 & 0.822824 & 0.828450  \\
  \noalign{\smallskip}\hline\noalign{\smallskip}
  3 & 17 & 6 & -0.9  & 0.822824 & 0.839165  \\
  \noalign{\smallskip}\hline\noalign{\smallskip}
  4 & 31 & 6 & -0.97 & 0.760157 & 0.762785  \\
  \noalign{\smallskip}\hline\noalign{\smallskip}
  4 & 31 & 6 & -0.95 & 0.760157 & 0.764851  \\
  \noalign{\smallskip}\hline\noalign{\smallskip}
  4 & 31 & 6 & -0.9  & 0.760157 & 0.771819  \\
  \noalign{\smallskip}\hline\noalign{\smallskip}
  3 & 26 & 8 & -0.97 & 0.885791 & 0.887931  \\
  \noalign{\smallskip}\hline\noalign{\smallskip}
  3 & 26 & 8 & -0.95 & 0.885791 & 0.890171  \\
  \noalign{\smallskip}\hline\noalign{\smallskip}
  3 & 26 & 8 & -0.9  & 0.885791 & 0.914420  \\
  \noalign{\smallskip}\hline\noalign{\smallskip}
  4 & 56 & 8 & -0.97 & 0.838596 & 0.840453  \\
  \noalign{\smallskip}\hline\noalign{\smallskip}
  4 & 56 & 8 & -0.95 & 0.838596 & 0.842071  \\
  \noalign{\smallskip}\hline\noalign{\smallskip}
  4 & 56 & 8 & -0.9  & 0.838596 & 0.849410  \\
  \hline\noalign{\smallskip}
\end{tabular}
\end{center}
\end{table}

\subsection{Case $t_1(y) \in [-1,\ell]$}
\label{sec:3.2}

This case is more subtle. In applications of \eqref{defin_f} we will have to optimize in two classes of real polynomials. We consider
\[ A(n,k,\ell):=\{ f(t)=A^2(t): \deg(f)=2k, A(t) \mbox{ has $k$ real zeros in } [\ell,t_{FL}] \}, \]
noting that every polynomial from $A(n,k,\ell)$ is decreasing in $[-1,\ell]$, nonnegative in $[\ell,t_{FL}]$,
and increasing in $[t_{FL},1]$ (recall that $t_{FL}=t_k^{0,1}$). Similarly, we use polynomials from the set
\[ B(n,k,s):=\{ g(t)=(t+1)B^2(t)(t-s): \deg(g)=2k, B(t) \mbox{ has $k-1$ real zeros in } [-1,s] \}, \]
where the parameter $s$ (close to $t_{FL}$) will be chosen in advance. We underline that
every polynomial from $B(n,k,\ell)$ is nonpositive in $[-1,s]$ and positive and increasing in $[s,1]$.

The next lemma sets an auxiliary parameter $m(C)$ after optimization in the set $A(n,\tau,\ell)$.

\begin{lemma} \label{lema_m1}
Let $f(t) \in A(n,k,\ell)$ and the positive integer $m$ be such that
\begin{equation} \label{mcond}
f_0|C|<f(\ell)+(m+1)f(t_{FL}).
\end{equation}
Then $t_{|C|-m}(y)<t_{FL}$.
\end{lemma}

\begin{proof}
Assume that $t_{|C|-m}(y) \geq t_{FL}$ for a contradiction, and consider \eqref{defin_f} for $C$, $y$, and $f(t)$. Then
\[ f_0|C| = \sum_{i=1}^{|C|} f(t_i(y)) \geq f(t_1(y))+\sum_{i=|C|-m}^{|C|} f(t_i(y)) \geq f(\ell)+(m+1)f(t_{FL}) \]
since $f(t)$ is decreasing in $[-1,\ell]$ (where $t_1(y)$ lies) and increasing in $[t_{FL},1]$ (where the inner products
$t_{|C|-m}(y) \leq \cdots \leq t_{|C|}(y)=\rho(C)$ are located), which contradicts to \eqref{mcond}.
\end{proof}

We define
\[ m(C):=\min \{ m: \exists f \in A(n,k,\ell) \mbox{ such that } f_0|C|<f(\ell)+(m+1)f(t_{FL}) \}. \]
Lemma \ref{last-n} implies that
\[ m(C) \geq n. \]

\begin{example} We have $m(C)=n=3$ for $(n,\tau,|C|)=(3,4,10)$, the first case, where the existence/nonexistence
of spherical 4-designs is undecided. Similarly, for $(n,\tau,|C|)=(4,4,15)$ we have $m(C)=n+1=5$.
More examples of $m(C)$ will be shown in Table \ref{tab:2} in Section \ref{sec:4} together with our new upper and lower bounds for $\rho(C)$.
\end{example}

\begin{lemma} \label{lema_m1-2}
Let $f(t) \in A(n,k,\ell)$ be such that $f_0|C|<f(\ell)+(m(C)+1)f(t_{FL})$. Then $t_{|C|-m(C)}(y) \leq s$,
where $s$ is the largest root of the equation
\[ f_0|C|-f(\ell)=(m(C)+1)f(t). \]
\end{lemma}

\begin{proof}
The choice of $m(C)$ shows that $m(C)f(t_{FL}) \leq f_0|C|-f(\ell)<(m(C)+1)f(t_{FL})$. Since
\[ f_0|C| \geq f(\ell)+(m(C)+1)f(t_{|C|-m(C)}) \]
as in the proof of Lemma \ref{lema_m1}, we obtain
\[ (m(C)+1)f(s)= f_0|C|-f(\ell) \geq (m(C)+1)f(t_{|C|-m(C)}). \]
Hence $f(s) \geq f(t_{|C|-m(C)})$ and the required inequality follows.
\end{proof}

Lemmas \ref{lema_m1} and \ref{lema_m1-2} imply that $t_{|C|-m(C)}(y)\leq s <t_{FL}$. We utilize this in a second optimization
dealing with the location of $t_{FL}$ between two inner products from $I(y)$. Lemmas \ref{last-n}, \ref{lema_m1} and \ref{lema_m1-2} imply that
\[ t_{|C|-m(C)}(y) \leq s < t_{FL} \leq t_{|C|-n+1}(y)=\rho(C). \]
 Therefore, there exist $j \in \{0,1,\ldots,m(C)-n\}$ such that
\begin{equation} \label{tfl-location}
t_{|C|-m(C)+j}(y) < t_{FL} \leq t_{|C|-m(C)+j+1}(y).
\end{equation}
This clarification of the location of $t_{FL}$ with respect to the points of $I(y)$ allows more precise estimations.

\begin{lemma} \label{lema_mj}
If $g(t) \in B(n,k,s)$, then $\rho(C) \geq m_{\ell,s}^{(j)}$, where $m_{\ell,s}^{(j)}$ is the largest root of the equation
\begin{equation} \label{mls} jg(t_{FL})+(m(C)-j)g(t)=g_0|C|. \end{equation}
\end{lemma}

\begin{proof}
We consider \eqref{defin_f} for $C$, $y$, and $g(t)$. Using $t_{|C|-m}(y) \leq s < t_{FL}$ and \eqref{tfl-location} we
estimate the values of $g(t)$ in the points of $I(y)$ as follows:
\[ g(t_i(y)) \leq 0, \ \ i=1,2, \ldots, |C|-m(C), \]
\[ g(t_i(y)) \leq g(t_{FL}), \ \ i=|C|-m(C)+1,|C|-m(C)+2,\ldots,|C|-m(C)+j, \]
\[ g(t_i(y)) \leq g(\rho(C)), \ \ i=|C|-m(C)+j+1,|C|-m(C)+j+2,\ldots,|C|-1,|C|. \]
Hence, we consecutively have
\begin{eqnarray*}
jg(t_{FL})+(m(C)-j)g(m_{\ell,s}^{(j)}) &=& g_0|C| =\sum_{i=1}^{|C|} g(t_i(y)) \\
&\leq& \sum_{i=|C|-m(C)+1}^{|C|-m(C)+j} g(t_i(y)) + \sum_{i=|C|-m(C)+j+1}^{|C|} g(t_i(y)) \\
&\leq& jg(t_{FL})+(m(C)-j)g(\rho(C)).
\end{eqnarray*}
This gives $g(m_{\ell,s}^{(j)}) \leq g(\rho(C))$, whence $\rho(C) \geq m_{\ell,s}^{(j)}$ by the definition of $m_{\ell,s}^{(j)}$
via \eqref{mls}.
\end{proof}

The relation \eqref{tfl-location} allows refinement of Lemma \ref{lema_m1-2} as well.

\begin{lemma} \label{lema_sj}
Let $f(t) \in A(n,k,\ell)$ be such that $f_0|C|<f(\ell)+(m(C)+1)f(t_{FL})$. Then $t_{|C|-m(C)}(y) \leq s^{(j)}$,
where $s^{(j)}$ is the largest root of the equation
\begin{equation} \label{sj}
f_0|C|=(j+1)f(t)+f(\ell)+(m(C)-j)f(t_{FL}).
\end{equation}
\end{lemma}

\begin{proof}
In \eqref{defin_f} with $f$, $C$ and $y$ we have $f(t_1(y)) \geq f(\ell)$,
\[ f(t_i(y)) \geq 0, \  \ i=2,3,\ldots, |C|-m(C)-1, \]
\[ f(t_i(y)) \geq f(t_{|C|-m(C)}(y), \ \ i=|C|-m(C),|C|-m(C)+1,\ldots,|C|-m(C)+j, \]
\[ f(t_i(y)) \geq f(t_{FL}), \ \ i=|C|-m(C)+j+1, |C|-m(C)+j+2,\ldots,|C|-1,|C|. \]
Using these, we consecutively obtain
\begin{eqnarray*}
f(\ell) + (j+1)f(s^{(j)}) + (m(C)-j)f(t_{FL}) &=& f_0|C| = \sum_{i=1}^{|C|} f(t_i(y)) \\
&\geq& f(\ell) + \sum_{i=|C|-m(C)}^{|C|-m(C)+j} f(t_i(y)) + \sum_{i=|C|-m(C)+j+1}^{|C|} f(t_i(y)) \\
&\geq& f(\ell) + (j+1)f(t_{|C|-m(C)}(y)) + (m(C)-j)f(t_{FL}).
\end{eqnarray*}
We obtain $f(s^{(j)}) \geq f(t_{|C|-m(C)}(y))$, whence $t_{|C|-m(C)}(y) \leq s^{(j)}$ by the definition of $s^{(j)}$
via \eqref{sj}.
\end{proof}

\subsection{A procedure for finding new lower bounds}
\label{sec:3.3}

In the case $t_1(y) \in [-1,\ell]$, for each fixed $j \in \{0,1,\ldots,m(C)-n\}$, we can start an iterative procedure with Lemmas \ref{lema_mj} and \ref{lema_sj} for obtaining consecutive improvements of $s^{(j)}$ and $m_{\ell,s}^{(j)}$. This procedure may
converge to some bounds or may be divergent which will mean nonexistence of designs with the corresponding parameters
(dimension, strength, and cardinality).

The better bound $\rho(C) \geq t_k^{0,\ell}$ when $t_1(y) \geq \ell$ allows starting a similar procedure with analogs of
Lemmas \ref{lema_mj} and \ref{lema_sj}
as the only difference will be the absence of $\ell$.

\begin{example} \label{ex2}
Considering again $(n,\tau,|C|)=(3,4,10)$ (recall that $m(C)=n=3$ in this case, i.e. $j=0$ only), we obtain for $\ell=-0.97$ that $\rho(C) \geq 0.724753$
if $t_1(y) \in [-1,-0.97]$ and $\rho(C) \geq 0.728787$ if $t_1(y) \geq -0.97$. Therefore, we have $\rho(C) \geq 0.724753$ in the worst case.

Similarly, for $(n,\tau,|C|)=(4,4,15)$ (note that now $m(C)=5$, i.e. $j=0,1$), we obtain for $\ell=-0.97$ that $\rho(C) \geq 0.625572$ if
$t_1(y) \in [-1,-0.97]$ and $\rho(C) \geq 0.627354$ if $t_1(y) \geq -0.97$ for $j=0$; $\rho(C) \geq 0.616854$ if $t_1(y) \in [-1,-0.97]$ and $\rho(C) \geq 0.619259$ if $t_1(y) \geq -0.97$ for $j=1$. Summarizing, we conclude that $\rho(C) \geq 0.616854$ in the worst case.
\end{example}

Further results of this procedure with new lower bounds are given in a table below along with upper bounds in the corresponding cases.

\section{Upper bounds}
\label{sec:4}

In this section we obtain upper bounds for the covering radius of spherical designs of fixed dimension, strength and cardinality.
Note that our approach works both for odd and even strengths.

\begin{remark}
Upper bounds for covering radius of spherical designs were presented by the authors in \cite{BS05} but were never
recorded in a journal publication.
\end{remark}

\subsection{General upper bounds}
\label{sec:4.1}

Using Lemma \ref{last-n} we write (\ref{defin_f}) for $y$, $C$ and $f(t)$,
$\deg(f) \leq \tau(C)$, as
\begin{equation} \label{main_tc}
nf(\rho(C))+\sum_{i=1}^{|C|-n} f(t_i(y)) = f_0|C|.
\end{equation}
The identity \eqref{main_tc} provides upper bounds for $\rho(C)$ as follows.

\begin{theorem} \label{lp-cr-designs}
{\rm (Linear programming upper bounds of the
covering radius of spherical designs)} Let $f(t)$, $\deg(f)\leq
\tau$, be a real polynomial which is nonnegative in $[-1,t_{FL}]$ and increasing in $[t_{FL},1]$. Then
for every $\tau$-design $C \subset \mathbb{S}^{n-1}$ we have
\[ \rho(C) \leq m_u, \]
where $m_u$ is the largest root of the equation $nf(t)=f_0|C|$.
\end{theorem}

\begin{proof}
Under the assumptions of the theorem we have $f(t_i(y))
\geq 0$ for $1 \leq i \leq |C|-n$. Also, Lemma 2.1 implies $f(t_i(y))
=f(\rho(C))$ for $|C|-n+1 \leq i \leq |C|$. Then (\ref{main_tc}) gives
the inequality $nf(\rho(C)) \leq f_0|C|$. Since $f(t)$ is increasing
in the interval $[t_{FL},+\infty)$, this implies our claim.
\end{proof}

The following theorem shows which kind of extremal polynomials
should be investigated.

\begin{theorem} \label{best-poly}
The best polynomials for use in Theorem \ref{lp-cr-designs}
are $f(t)=(t+1)^e A^2(t)$, where $\tau=2k-e$,
$e \in \{0,1\}$, $\deg(A)=k-e$ and $A(t)$ has
$k-e$ zeros in $[-1,t_{FL}]$.
\end{theorem}

\begin{proof}
Assume for a contradiction that $f(t)=B(t)D(t)$, where
$B(t) \geq 0$ has only double zeros in $[-1,1]$ and possibly $-1$ as a zero
and $D(t)$ is a nonconstant polynomial
which does not have zeros in $[-1,1]$. Then there exists (small enough)
$\varepsilon>0$ such that $g(t)=B(t)(D(t)-\varepsilon) \geq 0$ for $t \in [-1,1]$.

Denote by $t_1^{(u)}$ and $t_2^{(u)}$ the largest roots of the
equations $nf(t)=f_0|C|$ and $ng(t)=g_0|C|$, respectively. We will
show that there exists some $\varepsilon$ such that $t_2^{(u)}<
t_1^{(u)}$, i.e. the polynomial $g(t)$ gives better upper bound on
$\rho(C)$, thus obtaining a contradiction.

Since the derivative of $g(t)$ is
$g^\prime(t)=f^\prime(t)-\varepsilon B^\prime(t)>0$ for every
$t \in [t_{FL},1]$ and for small enough fixed $\varepsilon$, the polynomial $g(t)$
is strictly increasing (for such $\varepsilon$). Then the inequality
$t_2^{(u)}< t_1^{(u)}$ follows from $g(t_2^{(u)})< g(t_1^{(u)})$, which is
equivalent to
\[ \varepsilon<\frac{|C|(f_0-g_0)}{nB(t_1^{(u)})}. \]

We have $g_0=f_0-\varepsilon b_0$, where $b_0$ is the first
coefficient in the Gegenbauer expansion of $B(t)$. Since
$b_0=c_n\int_{-1}^1B(t)(1-t^2)^{(n-3)/2}dt>0$, we have
$g_0<f_0$. Therefore $\frac{|C|(f_0-g_0)}{nB(t_1^{(u)})}>0$ and
it is clear now that we can choose the necessary small enough $\varepsilon$.
\end{proof}

The best polynomials $A(t)$ still have to be found. We pursue this for $\tau=4$ below.

\subsection{Upper bounds for spherical 4-designs}
\label{sec:4.2}

We now find the optimal polynomials in Theorem \ref{lp-cr-designs} for $\tau=4$.

\begin{theorem} \label{ub-thm}
If $C \subset \mathbb{S}^{n-1}$ is a spherical 4-design, then
\[ \rho(C) \leq  u(a_0,b_0), \]
where the function $u(a,b)$ and the (optimal) parameters $a_0$ and $b_0$ are defined in the proof.
\end{theorem}

\begin{proof}
Let $f(t)=(t^2+at+b)^2$, where $a$ and $b$ are parameters to be optimized. Then
\[ f_0=b^2+\frac{a^2+2b}{n}+\frac{3}{n(n+2)} \]
and we have to consider the inequality
\[ (\rho(C))^2+a\rho(C)+b)^2 \leq
\frac{|C|}{n}\left(b^2+\frac{a^2+2b}{n}+\frac{3}{n(n+2)}\right).\]
We obtain
\[ \rho(C) \leq u(a,b) :=-\frac{a}{2}+\frac{1}{2}\left(a^2-4b+
4\sqrt{\frac{|C|}{n}\left(b^2+\frac{a^2+2b}{n}+\frac{3}{n(n+2)}\right)}
\hspace*{0.4mm}\right)^{1/2},
\] where we have to minimize the function $u(a,b)$. Using the
standard approach via partial derivatives we obtain the following
optimal values of $a$ and $b$
\begin{eqnarray*}
b_0 & :=&
\frac{3n(n+1)-(n+2)|C|+\sqrt{n(n-1)\left[2(n+2)|C|-3n(n+3)\right]}}
{n(n+2)\left(|C|-2n\right)}, \\
a_0 & :=& \frac{nb_0+1}{n}\sqrt{\frac{|C|(nb_0+1)-n^2b_0}{nb_0+2}}.
\end{eqnarray*}
The corresponding bound $\rho(C) \leq u(a_0,b_0)$ is rather long to be stated here.
\end{proof}

\begin{example}
Looking again in the first open case $(n,\tau,|C|)=(3,4,10)$, we obtain
(for $b_0 = \frac{\sqrt{69}-7}{30}$ and $a_0 = $ $\frac{(3+\sqrt{69})\sqrt{45+10\sqrt{69}}}{150}$)
the upper bound $\rho(C) \leq u(a_0,b_0) \approx 0.754443$.
\end{example}

In the table below we present new lower and upper bounds for the covering radius of
spherical 4-designs in dimensions $3 \leq n \leq 10$ and cardinalities $|C|=D(n,4)+1$ and
$D(n,4)+2$. The lower bounds are truncated after the sixth digit while the upper bounds are
rounded from in the fourth digit.

\begin{table}
\caption{}
\label{tab:2}
\begin{center}
\begin{tabular}{|c|c|c|c|c|c|}
  \hline
  Dimension & Cardinality &  $m(C)$  & Fazekas- & New lower & Upper bound, \\
  $n$ & $|C|$ & & Levenshtein & bound,  & Theorem \ref{ub-thm} \\
&&& lower bound & $\ell=-0.97$ & \\
&&& $\rho(C) \geq t_2^{0,1}$ && \\
  \noalign{\smallskip}\hline\noalign{\smallskip}
  3 & 10 &  3 & 0.689897 & 0.724753 & 0.7545 \\
  \noalign{\smallskip}\hline\noalign{\smallskip}
  3 & 11 &  4 & 0.689897 & 0.694717 & 0.7794 \\
  \noalign{\smallskip}\hline\noalign{\smallskip}
  4 & 15 &  5 & 0.607625 & 0.616854 & 0.6918 \\
  \noalign{\smallskip}\hline\noalign{\smallskip}
  4 & 16 &  5 & 0.607625 & 0.610537 & 0.7072 \\
  \noalign{\smallskip}\hline\noalign{\smallskip}
  5 & 21 &  7 & 0.546918 & 0.550012 & 0.6503 \\
  \noalign{\smallskip}\hline\noalign{\smallskip}
  5 & 22 &  8 & 0.546918 & 0.548132 & 0.6604 \\
  \noalign{\smallskip}\hline\noalign{\smallskip}
  6 & 28 & 10 & 0.500000 & 0.501717 & 0.6198 \\
  \noalign{\smallskip}\hline\noalign{\smallskip}
  6 & 29 & 10 & 0.500000 & 0.501288 & 0.6269 \\
  \noalign{\smallskip}\hline\noalign{\smallskip}
  7 & 36 & 13 & 0.462475 & 0.463455 & 0.5960 \\
  \noalign{\smallskip}\hline\noalign{\smallskip}
  7 & 37 & 13 & 0.462475 & 0.462961 & 0.6012 \\
  \noalign{\smallskip}\hline\noalign{\smallskip}
  8 & 45 & 16 & 0.431662 & 0.431663 & 0.5766 \\
  \noalign{\smallskip}\hline\noalign{\smallskip}
  8 & 46 & 17 & 0.431662 & 0.432103 & 0.5805 \\
  \noalign{\smallskip}\hline\noalign{\smallskip}
  9 & 55 & 20 & 0.405827 & 0.405915 & 0.5602 \\
  \noalign{\smallskip}\hline\noalign{\smallskip}
  9 & 56 & 21 & 0.405827 & 0.406039 & 0.5633 \\
  \noalign{\smallskip}\hline\noalign{\smallskip}
 10 & 66 & 25 & 0.383795 & 0.383922 & 0.5461 \\
  \noalign{\smallskip}\hline\noalign{\smallskip}
 10 & 67 & 25 & 0.383795 & 0.383972 & 0.5486 \\
  \noalign{\smallskip}\hline
\end{tabular}
\end{center}
\end{table}

\subsection{Upper bounds for antipodal 3- and 5-designs}
\label{sec:4.3}

A spherical design $C$ is called antipodal if $C=-C$. Since the set $I(y)$ is symmetric for antipodal designs,
only small modifications in Lemma \ref{last-n}, the identity \eqref{main_tc} and Theorem \ref{lp-cr-designs} are needed.
Indeed, we have
\[ t_i=t_{|C|-i+1}, \ i=1,2,\ldots,n, \]
in Lemma \ref{last-n} and the equation in Theorem \ref{lp-cr-designs} becomes $2nf(t)=f_0|C|$. Recall that the Fazekas-Levenshtein
bound \eqref{FL_bound} for $(2k-1)$-designs is $t_{FL}=t_k^{0,0}=t_{k,k}$.

The upper bounds for antipodal designs are easier for $\tau=3$ and
$\tau=5$. We are able to obtain an explicit bound in these cases
for all dimensions and cardinalities.

\begin{theorem} \label{cr-3-anti}
If $C$ is an antipodal 3-design, then
\[ t_{FL}=\frac{1}{\sqrt{n}} \leq \rho(C) \leq \frac{1}{n}\sqrt{\frac{|C|}{2}}. \]
\end{theorem}

\begin{proof}
We use $f(t)=t^2$ in Theorem \ref{lp-cr-designs} for antipodal designs.
Then $f_0=1/n$ and we have to find the largest root of the
equation $2n^2t^2=|C|$.
\end{proof}

\begin{theorem} \label{cr-5-anti}
If $C$ is an antipodal 5-design, then
\[ t_{FL}=\left(\frac{3}{n+2} \right)^{1/2} \leq \rho(C) \leq
\left(\frac{1}{n}+\frac{1}{n}\sqrt{\frac{(n-1)(|C|-2n)}{n(n+2)}}\hspace*{0.4mm}
\right)^{1/2}. \]
\end{theorem}

\begin{proof}
We apply Theorem \ref{lp-cr-designs} polynomial $f(t)=(t^2-a)^2$, where the
optimal value of $a \in [-t_{FL},t_{FL}]$ will be determined. We
have $f(t_i) \geq 0$ for $n+1 \leq i \leq |C|-n$. Then
(\ref{main_tc}) gives
\[ 2nf(\rho(C))=n[f(\rho(C))+f(-\rho(C))] \leq f_0|C|=\left(a^2-\frac{2a}{n}+\frac{3}{n(n+2)}\right)|C|. \]
Since $f(t)$ is positive and increasing in $[t_{FL},+\infty)$,
this implies \[ (\rho(C))^2 \leq
a+\left(\frac{|C|(a^2-2a/n+3/n(n+2))}{2n}\right)\,^{1/2}.\] The
right-hand side is minimized for
\[ a=\frac{1}{n}-2\sqrt{\frac{n-1}{n(n+2)(|C|-2n)}} \] which gives the
desired bound. We used Maple to deal with this case.
\end{proof}

{\bf{Acknowledgements}}
The authors thank to Arseniy Akopyan and Nikolai Nikolov for helpful discussions.
\thanks{The research of the first author was partially supported by Bulgarian NSF under project KP-06-N32/2-2019.
The research of the second author was supported, in part, by the National Scientific Program
''Information and Communication Technologies for a Single Digital Market in Science, Education and Security (ICTinSES)'',
NIS-3317, financed by the Bulgarian Ministry of Education and Science.}

% Authors must disclose all relationships or interests that
% could have direct or potential influence or impart bias on
% the work:
%
%\section*{Conflict of interest}
%
%The authors declare that they have no conflict of interest.


\begin{thebibliography}{99}

\bibitem{BD1}
E. Bannai, R. Damerell, Tight spherical designs I,
{\it J. Math. Soc. Japan} 31, 199-207 (1979).

\bibitem{BD2}
E.Bannai, R.M.Damerell, Tight spherical designs II,
{\it J. London Math. Soc.} 21, 13-30 (1980).

\bibitem{BMV04}
E. Bannai, A. Minemasa, B. Venkov,
The nonexistence of certain tight spherical designs,
{\it Algebra i Analiz}, 16(4), 1-23 (2004) (in Russian);
English translation in {\it St. Petersburg Math. J.} 16, 609-625 (2005).

%\bibitem{BS81}
%E. Bannai, N. J. A. Sloane,
%Uniqueness of certain spherical codes,
%{\it Can. J. Math.} 33, 437-449 (1981).

\bibitem{BRV13}
A. Bondarenko, D. Radchenko, M. Viazovska,
Optimal asymptotic bounds for spherical designs,
{\it Ann. Math.} 178(2), 443-452 (2013).

\bibitem{BRV15}
A. Bondarenko, D. Radchenko, M. Viazovska,
Well-separated spherical designs,
{\it Constr. Approx.} 41, 93-112 (2015).

\bibitem{BBD}
S. Boumova,  P. Boyvalenkov, D. Danev,
Necessary conditions for existence of some designs in polynomial metric spaces,
{\it Europ. J. Combin.} 20, 213-225 (1999).

\bibitem{BBKS}
S. Boumova, P. Boyvalenkov, H. Kulina, M. Stoyanova,
Polynomial techniques for investigation of spherical designs,
{\it Designs, Codes and Cryptography}, 51(3), 275-288, (2009).

%\bibitem{Boy}
%P. G. Boyvalenkov, Computing distance distributions of spherical designs,
%{\it Lin. Alg. Appl.}, 226/228, 277-286 (1995).

\bibitem{Boy1}
P. G. Boyvalenkov,
Extremal polynomials for obtaining bounds for spherical codes and designs,
{\it Discr. Comp. Geom.} 14, 167-183 (1995).

%\bibitem{BDN}
%P. Boyvalenkov, D. Danev, S. Nikova,
%Nonexistence of %certain spherical designs of odd strengths and cardinalities,
%{\it Discr. Comp. Geom.} 21, 143-156 (1999).

\bibitem{BDHSS-DCC2019}
P. Boyvalenkov, P. Dragnev, D. Hardin, E. Saff, M. Stoyanova,
On spherical codes with inner products in prescribed interval,
{\it Designs, Codes and Cryptography} 87, 299-315 (2019).

\bibitem{BS05}
P. Boyvalenkov, M. Stoyanova,
Upper bounds on the covering radius of spherical designs,
in Proc. Intern. Workshop on Optimal Codes, Pamporovo (Bulgaria), June 17-23, 53-58 (2005).

\bibitem{DGS}
P. Delsarte, J.-M. Goethals, J. J. Seidel,
Spherical codes and designs,
{\it Geom. Dedicata} 6, 363-388 (1977).

\bibitem{FL}
G. Fazekas, V. I. Levenshtein,
On upper bounds for code distance and covering radius of designs in polynomial metric spaces,
{\it J. Comb. Theory A}, 70, 267-288 (1995).

%\bibitem{HS}
%R. H. Hardin, N. J. A. Sloane,
%McLaren's improved snub cube and other new spherical designs in three dimensions,
%{\it Discr. Comp. Geom.} 15, 429-441 (1996).

\bibitem{Lev2}
V. I. Levenshtein,
Universal bounds for codes and designs,
Chapter 6 (499-648) in {\it Handbook of Coding Theory},
Eds. V.Pless and W.C.Huffman, Elsevier Science B.V., (1998).

\bibitem{NN}
S. Nikova, V. Nikov,
Improvement of the Delsarte bound for t-designs when it is not the best bound possible,
{\it Des., Codes Crypt.} 28, 201-222 (2003).

\bibitem{NV12}
G. Nebe, B. Venkov,
On tight spherical designs,
{\it Algebra i Analiz}, 24(3), 163-171 (2012) (in Russian);
English translation in {\it St. Petersburg Math. J.} 24, 485-491 (2013).

%\bibitem{Rez}
%B. Reznick, Some constructions of spherical 5-designs,
%{\it Lin. Alg. Appl.} 226/228, 163-196 (1995).

\bibitem{Sze} G. Szeg\"{o},
{\it Orthogonal polynomials},
{AMS Col. Publ.}, vol. 23, Providence, RI (1939).

\bibitem{Yud95}
V. A. Yudin,
Coverings of a sphere, and extremal properties of orthogonal polynomials,
{\it Discrete Math. Appl.} 5, 371-379 (1995).

\bibitem{Yud97}
V. A. Yudin,
Lower bounds for spherical designs,
{\it Izv. RAN, Ser. Mat.}, 61, 213-223 (1997) (in Russian);
English translation in {\it Izv. Math.}, 61, 673-683 (1997).

\end{thebibliography}
\end{document}